\documentclass[ ]{elsarticle}

\journal{Journal of Computational and Applied Mathematics}

\usepackage[letterpaper,margin=1in]{geometry}

\usepackage[utf8]{inputenc}

\usepackage{microtype}

\usepackage{amsmath,amssymb,amsthm,mathtools,amscd}
\usepackage{bm}
\usepackage{enumitem}
\usepackage{IEEEtrantools}
\usepackage{tikz}
\usepackage{graphicx}
\usepackage{subcaption}
\usepackage{caption}
\usepackage{epstopdf}
\usepackage[section]{algorithm}
\usepackage{algpseudocode}
\usepackage{xifthen}

\DeclareMathOperator*{\argmin}{arg\,min}

\newcommand{\cA}{\mathcal{A}}
\newcommand{\cB}{\mathcal{B}}
\newcommand{\cE}{\mathcal{E}}
\newcommand{\cG}{\mathcal{G}}
\newcommand{\cI}{\mathcal{I}}
\newcommand{\cV}{\mathcal{V}}
\newcommand{\half}{\frac{1}{2}}

\newcommand{\Reals}{\mathbb{R}}

\newcommand{\mforall}{\forall\,}

\newcommand{\id}{\operatorname{id}}

\newcommand{\dvg}{\operatorname{div}}
\newcommand{\Grad}{G}
\newcommand{\Div}{G^*}
\newcommand{\vsp}{V}
\newcommand{\esp}{W}

\newcommand{\vone}{\mathbf{1}}
\newcommand{\eone}{\mathbf{1}}

\newcommand{\dif}{\,\mathrm d}
\newcommand{\eff}[1][]{e_{\text{ff}\ifthenelse{\isempty{#1}}{}{,#1}}}
\newcommand{\norm}[1]{\left\lVert#1\right\rVert}
\newcommand{\ip}[3][]{\left(#2,#3\right)\ifthenelse{\isempty{#1}}{}{_{#1}}} 
\newcommand{\avedg}[2]{\{\!\!\{#1\}\!\!\}_{#2}}

\newcommand{\Poincare}{Poincar\'{e}}

\let\originalleft\left
\let\originalright\right
\renewcommand{\left}{\mathopen{}\mathclose\bgroup\originalleft}
\renewcommand{\right}{\aftergroup\egroup\originalright}

\allowdisplaybreaks

\newtheorem{theorem}{Theorem}[section]
\newtheorem{lemma}[theorem]{Lemma}

\newtheorem*{remark*}{Remark}

\numberwithin{equation}{section}

\graphicspath{{PicPNG/}{PicPDF/}{PicEPS/}}

\newcommand{\red}[1]{#1}

\begin{document}

\makeatletter
\def\@author#1{\g@addto@macro\elsauthors{\normalsize%
    \def\baselinestretch{1}%
    \upshape\authorsep#1\unskip\textsuperscript{%
      \ifx\@fnmark\@empty\else\unskip\sep\@fnmark\let\sep=,\fi
      \ifx\@corref\@empty\else\unskip\sep\@corref\let\sep=,\fi
    }%
    \def\authorsep{\unskip,\space}%
    \global\let\@fnmark\@empty
    \global\let\@corref\@empty  
    \global\let\sep\@empty}%
  \@eadauthor={#1}
}
\makeatother

\begin{frontmatter}

\title{Adaptive Aggregation on Graphs\tnoteref{ztitlenote}}

\tnotetext[ztitlenote]{This work is supported in part by NSF through
  grants DMS-1522615, DMS-1720114, and by Lawrence Livermore National
  Lab through subcontract B614127.}
  
\author{Wenfang Xu}

\address{Department of Mathematics, The Pennsylvania State University,
  University Park, PA 16802, USA\, {Email:} wxx107@psu.edu}

\author{Ludmil T. Zikatanov\corref{cor1}}

\address{Department of Mathematics, The Pennsylvania State University,
  University Park, PA 16802, USA\, {Email:} ludmil@psu.edu\\
  Institute for Mathematics and Informatics, Bulgarian Academy of
  Sciences, Sofia, Bulgaria}

\begin{abstract}
We generalize some of the functional (hypercircle) a
posteriori estimates from finite element settings to general graphs
or Hilbert space settings. Several theoretical results in regard to
the generalized a posteriori error estimators are provided. We use
these estimates to construct aggregation based coarse spaces for
graph Laplacians. The estimator is used to assess the quality of an
aggregation adaptively. Furthermore, a reshaping based algorithm is
tested on several numerical examples.
\end{abstract}

\end{frontmatter}

\section{Introduction}
\red{The hypercircle identity was first introduced in
  \cite{1947PragerW_SyngeJ-aa} to study approximations to elastic
  problems.} Then the so-called hypercircle methods are established
and studied in a posteriori error estimates for finite element methods
\red{(see, for example, \cite{MR3076038}, for a comprehensive survey
  of the main results)}. Some pioneering works on the hypercircle
methods are \cite{1975Ladeveze, 1999DestuynderMetivet}.

The functional a posteriori error estimates are established in
\cite{MR2458008} for general elliptic problems defined from dual
operators between Banach and Hilbert spaces. The estimation is used,
for example, in \cite{MR2502931} to bound the conforming error in
discontinuous Galerkin approximations of elliptic problems (see also
\cite{MR2526991,MR2813186}). In \cite{MR2425501} the hypercircle
method is used to construct a posteriori error estimates for the
obstacle problem (see also~\cite{MR2095603}). The formulation of the
hypercircle identity and error estimates, however, arises naturally in
Hilbert space settings and can be applied to graph Laplacians, a fact
which we exploit in what follows.

A multilevel graph coarsening scheme based on matching (that is,
collapsing adjacent vertices into aggregates) is studied in
\cite{MR1639073} and later used in several AMG methods. For example,
the multigrid method proposed in \cite{MR1964291} uses matching on the
graph of the stiffness matrix to solve convection--diffusion
equations. The AGMG method (AGgregation-based algebraic MultiGrid) in
\cite{MR2914316} employs a pairwise aggregation algorithm by matching
which minimizes a strength function. In \cite{MR3264812} matching
techniques which optimize matrix invariants were studied. In our work
we use matching to generate multilevel hierarchies for solving the
graph Laplacian. We point out that other coarsening techniques exist
in the literature, for example the compatible relaxation algorithm
\cite{MR1760015, MR2065813, MR2114296, MR2652083}.

In this paper combinatorial graphs are considered and we are
interested in approximations of the associated graph Laplacian matrix
from coarse subspaces. We propose Raviart-Thomas-like coarsening
schemes for both the vertex and edge spaces defined on graph
aggregations. It can be shown that the functional a posteriori error
estimates naturally apply to this setting and we provide a short proof
of the estimation, inspired by the works in \cite{1975Ladeveze,
  1999DestuynderMetivet, 1947PragerW_SyngeJ-aa, MR2458008}. The
estimator is minimized by an inter-leaved method to achieve reliable
bound of the error \cite{MR2817075}. Lastly we propose a reshaping
algorithm that generates aggregations adaptively. This algorithm can
be used together with aggregation coarsening methods \cite{MR1639073}
to form multilevel hierarchies for the graph Laplacian \red{(see
  \cite{MR1835471} for an analysis on the convergence of an algebraic
  multigrid method based on smoothed aggregation)}. We point out that
this is the first multilevel hierarchy that we know of that is
dependent on the right-hand side of the system. Several numerical
experiments are given at the end.

\section{Preliminaries and notation}
In this section we introduce the notations and preliminaries. Consider
a combinatorial graph $\cG = (\cV,\cE)$, where
$\cV = \{1,2,\cdots,n\}$ is the set of vertices and $\cE$ is the set
of edges containing pairs of the form $(i,j)$, where
$i,j\in\{1,\cdots,n\}$. Let $\vsp = \mathbb{R}^{\lvert\cV\rvert}$ be
the vertex space of $\cG$ and $\esp = \mathbb{R}^{\lvert\cE\rvert}$ be
the edge space. We consider $A\in \mathbb{R}^{n\times n}$ defined via
the bilinear form
\begin{equation*}
  (Au,v) = \sum_{(i,j)\in \cE}-a_{ij}(u_i-u_j)(v_i-v_j), \quad
  \mforall u,v\in \vsp.
\end{equation*}
Here the sum runs over all edges $e=(i,j)\in \cE$. The resulting
matrix is known as the (weighted) Graph Laplacian of $\cG$. Such
bilinear forms correspond to the continuous (conforming) finite
element as well as the mixed finite element discretizations with
lumped mass of a scalar elliptic equation. In our work, we generalize
some of the results from finite element a posteriori analysis to the
case of general graphs. In particular, we are interested in good
approximations of the above bilinear form on a smaller subspace, that
is, good approximations to the solution of the problem
\begin{equation}\label{eq:problem}
  Au = f.
\end{equation}

We define operators $\Grad\colon\vsp\to\esp$ and $D\colon\esp\to\esp$
as follows.
\begin{align*}
  (\Grad v)_e & = v_{\text{head}}-v_{\text{tail}}, \quad \mforall v\in\vsp, \\
  (D\bm\tau)_e & = a_e\bm\tau_e, \quad a_e = -a_{ij}, \quad \mforall\bm\tau\in\esp,
\end{align*}
where $e=(i,j)$ is any edge of the graph. Here the ``head'' and
``tail'' are predetermined for each edge. Given $\Grad$ and $D$ we can
rewrite the graph Laplacian as $(Au,v)=(D\Grad u,\Grad v)$. If $D$ is
chosen to be the identity map we obtain the so called standard graph
Laplacian. We further denote by $\Div$ the adjoint of $\Grad$ with
respect to the $\ell^2$-inner products on $\vsp$ and $\esp$, that is,
\begin{equation*}
  (\Grad u,\bm\tau)_{\esp}=(u,\Div\bm\tau)_{\vsp},
  \quad \mforall u\in\vsp,\ \mforall\bm\tau\in\esp.
\end{equation*}

Some examples from finite element methods naturally arise in the form
of graph Laplacian.

\subsection{Example I: standard finite elements}
Let $\Omega$ be a domain in $\Reals^n$. We consider Poisson's equation
with Neumann boundary condition.
\begin{alignat*}{2}
  -\dvg(a\nabla u)     & = f, \qquad && \text{in }\Omega;\\
  \nabla u\cdot\mathbf n & = 0, \qquad && \text{on }\partial\Omega.  
\end{alignat*} 
The weak formulation of the problem is then to find $u\in H^1(\Omega)$
such that
\begin{equation*}
  (Au,v)=(f,v),\quad\mforall v\in H^1(\Omega),
\end{equation*}
where
\begin{equation*}
(Au,v)=\int_\Omega a\nabla u\nabla v\dif x.
\end{equation*}
If the discretization by continuous piecewise linear elements is
considered, $A$ has a matrix representation
\begin{equation*}
(Au,v)=\sum_{(i,j)\in \cE}-a_{ij}(u_i-u_j)(v_i-v_j).
\end{equation*}
Clearly the above bilinear form is a weighted graph Laplacian.

\subsection{Example II: mixed finite elements with lumped mass}\label{ex:II}
We consider again Poisson's equation, but in mixed formulation.
\begin{IEEEeqnarray*}{rCl'l}
  (a^{-1}\bm\sigma,\bm\tau)+(\dvg\bm\tau,u) & = & 0,      & \mforall\bm\tau\in H(\dvg,\Omega); \\
  (\dvg\bm\sigma,v)                         & = & -(f,v), & \mforall v\in L^2(\Omega).
\end{IEEEeqnarray*}
Define the operator $A$ on $L^2(\Omega)$ as
\begin{equation*}
(Au,v)=(\dvg\bm\sigma,v).
\end{equation*}
For discretization let $\esp_h\subset H(\dvg,\Omega)$ be the space of
piecewise first order polynomials with continuous normal components
across the edges, and $S_h\subset L^2(\Omega)$ be the space of
piecewise constant functions. If we use ``mass lumping'' for the term
$(a^{-1}\bm\sigma,\bm\tau)$ then we get the discretized
$A\colon$
\begin{equation*}
  (Au,v)=\sum_{e=T_+\cap T_-}a_e(u_{T_+}-u_{T_-})(v_{T_+}-v_{T_-}),
\end{equation*}
which again is a graph Laplacian.

\section{Generalized functional (hypercircle) a posteriori error
  estimates}\label{sec:functional-estimates}
Here we present two lemmata in general settings that can be used in
functional a posteriori estimates. We provide abstract formulations
and proofs of the two results so that they can be used on graphs, and
even in more general Hilbert space settings.

Let $\vsp$ and $\esp$ be Hilbert spaces and suppose
$\Grad\colon\vsp\to\esp$ is an injective operator. We assume that the
Hilbert adjoint of $\Grad$, $\Div\colon\esp\to\vsp$ defined via
\begin{equation*}
  (\Div\bm\tau,v)_\vsp=(\bm\tau,\Grad v)_\esp, \quad
  \mforall v\in\vsp,\ \mforall\bm\tau\in\esp
\end{equation*}
is surjective with closed range. Note that for finite dimensional
cases, the closed range assumption is automatically true.

Let $D\colon\esp\to\esp$ be an operator that is \red{symmetric
  positive definite on $\esp$}. Define $A \coloneqq \Div D\Grad$ which
is positive definite. The problem we are interested in is to find
$u\in\vsp$, such that
\begin{equation}\label{eq:variational}
  (Au,v)_\vsp=f(v),\quad\mforall v\in\vsp.
\end{equation}
Here $f\in \vsp'$, the dual space of $\vsp$. Since $A$ is positive
definite, there exists some $C_P$, the \Poincare's constant, such that
\begin{equation*}
  C_P\norm{v}_\vsp\leq\norm{v}_A, \quad \mforall v\in\vsp.
\end{equation*}
In what follows, we will need the space $\esp(g)$, which for a fixed
$g\in \vsp'$ is defined to be
\begin{equation}\label{eq:q}
  \esp(g) = \{ \bm\tau\in\esp \mid (\bm\tau,\Grad v)_\esp = g(v),
  \mforall v\in\vsp \}.
\end{equation}

The first result is the following lemma.
\begin{lemma}[Prager and Synge \cite{1947PragerW_SyngeJ-aa}]
  \label{lem:Prager-Synge}
  Let $u$ be the solution to \eqref{eq:variational}. Then for any
  $\bm\tau\in\esp(f)$ and any $v\in\vsp$, the following identity
  holds.
  \begin{equation}\label{eq:Prager-Synge}
    \norm{u-v}^2_A + \norm{D\Grad u-\bm\tau}^2_{D^{-1}} = \norm{D\Grad v-\bm\tau}^2_{D^{-1}}.
  \end{equation}
\end{lemma}
\begin{proof}
  We have the following from the definition of $A$, as well as
  \eqref{eq:variational} and \eqref{eq:q}.
  \begin{align*}
    & \norm{D\Grad v-\bm\tau}^2_{D^{-1}}-\norm{D\Grad u-\bm\tau}^2_{D^{-1}}
      = (D\Grad v,D\Grad v)_{D^{-1}}-2(\Grad v,\bm\tau)_\esp+(\bm\tau,\bm\tau)_{D^{-1}} \\
    & \hspace{7ex} -(D\Grad u,D\Grad u)_{D^{-1}}+2(\Grad u,\bm\tau)_\esp-(\bm\tau,\bm\tau)_{D^{-1}} \\
    & \hspace{4ex} = \norm{v}_A^2-2f(v)-\norm{u}_A^2+2f(u)=\norm{v}_A^2-2(Au,v)_\vsp-\norm{u}_A^2+2(Au,u)_\vsp \\
    & \hspace{4ex} = \norm{v}_A^2-2(Au,v)_\vsp-\norm{u}_A^2+2\norm{u}_A^2=\norm{u-v}_A^2. \qedhere
  \end{align*}
\end{proof}

We may take $v\in\vsp$ to be an approximation to the solution $u$ in
\eqref{eq:variational}. Using the identity from lemma
\ref{lem:Prager-Synge} we can obtain the following important result
(S. Repin \cite{MR2458008}, Ladev{\`e}ze \cite{1975Ladeveze} and
Destuynder and M\'etivet \cite{1999DestuynderMetivet}).
\begin{lemma}\label{lem:Repin}
  Let $u$ be the solution to the variational problem
  \eqref{eq:variational}. Assume that $\bm\phi\in\esp$ is
  arbitrary. Then the following inequality holds for all $v\in\vsp$.
  \begin{equation}\label{eq:functional-estimate}
    \norm{u-v}_A \leq \norm{D\Grad  v-\bm\phi}_{D^{-1}} + C_P^{-1}\norm{\Grad ^*\bm\phi-f}_\vsp.
  \end{equation}
\end{lemma}
\begin{proof}
  By lemma \ref{lem:Prager-Synge}, for any $\bm\tau\in\esp(f)$ we have
  that
  \begin{equation*}
    \norm{u-v}^2_A\leq\norm{u-v}^2_A+\norm{D\Grad u-\bm\tau}^2_{D^{-1}}=\norm{D\Grad v-\bm\tau}^2_{D^{-1}}.
  \end{equation*}
  The triangle inequality then gives, for any $\bm\phi\in\esp$,
  \begin{equation*}
    \norm{u-v}_A\leq\norm{D\Grad v-\bm\phi}_{D^{-1}}+\norm{\bm\phi-\bm\tau}_{D^{-1}}.
  \end{equation*} 
  Because the inequality holds for all $\bm\tau\in\esp(f)$, we can take
  the infimum with respect to $\bm\phi$ and get
  \begin{equation*}
    \norm{u-v}_A\leq\norm{D\Grad v-\bm\phi}_{D^{-1}}+\inf_{\bm\tau\in\esp(f)}\norm{\bm\phi-\bm\tau}_{D^{-1}}.
  \end{equation*}
  We now estimate the last term on the right hand side. Let $z\in\vsp$
  be such that
  \begin{equation}\label{eq:auxiliary}
    (Az,w)_\vsp=f(w)-(\bm\phi,\Grad w)_\esp, \quad \mforall w\in \vsp.
  \end{equation}
  Since $A$ is symmetric positive definite such $z$ is guaranteed to
  exist. Define $\widetilde{\bm\tau}=D\Grad z+\bm\phi\in\esp$. We then
  have, by that $A=\Div D\Grad $ and \eqref{eq:auxiliary},
  \begin{equation*}
    (\widetilde{\bm\tau},\Grad w)_\esp = (D\Grad z,\Grad w)_\esp +
    (\bm\phi,\Grad w)_\esp = (Az,w)_\vsp + (\bm\phi,\Grad w)_\esp = f(w).
  \end{equation*} 
  This proves that $\widetilde{\bm\tau}\in\esp(f)$. Hence
  \begin{equation*}
    \inf_{\bm\tau\in\esp(f)}\norm{\bm\phi-\bm\tau}_{D^{-1}}\leq\norm{\bm\phi-\widetilde{\bm\tau}}_{D^{-1}}=
    \norm{D\Grad z}_{D^{-1}}=\norm{z}_A.
  \end{equation*}
  On the other hand, we apply \eqref{eq:auxiliary} again to obtain
  \begin{align*}
    \norm{z}_A & = \frac{\norm{z}^2_A}{\norm{z}_A} = \frac{(Az,z)_\vsp}{\norm{z}_A}
                 = \frac{(f-\Div\bm\phi,z)_\vsp}{\norm{z}_A} \\
               & \leq\norm{f-\Div\bm\phi}_\vsp\frac{\norm{z}_\vsp}{\norm{z}_A} \leq
                 C_P^{-1}\norm{\Div\bm\phi-f}_\vsp,
  \end{align*} 
  which concludes the proof.
\end{proof}

\red{
\begin{remark*}
  Although the original proof was found in the aforementioned papers,
  the anonymous referee provided a shorter and inspiring proof. We
  present it here. Since $A$ is symmetric positive definite,
  $\ip[A]{\cdot}{\cdot}$ defines an inner product. We get
  \begin{align*}
    \lVert u-v \rVert_A
    & = \sup_{\lVert z\rVert_A = 1}(u-v,z)_A
      = \sup_{\lVert z\rVert_A = 1}(A(u-v),z)_\vsp
      = \sup_{\lVert z\rVert_A = 1}(f-\Grad^*D\Grad v,z)_\vsp \\
    & = \sup_{\lVert z\rVert_A = 1}\{(f-G^*\bm\phi,z)_\vsp
      + (G^*(\bm\phi-DGv),z)_\vsp\}.
  \end{align*}
  We have for the first term
  \begin{equation*}
    (f-G^*\bm\phi,z)_\vsp \leq \norm{f-G^*\bm\phi}_\vsp\norm{z}_\vsp
    \leq \norm{f-G^*\bm\phi}_\vsp\cdot C_P^{-1}\norm{z}_A
    = C_P^{-1}\norm{G^*\bm\phi-f}_\vsp,
  \end{equation*}
  and for the second term
  \begin{align*}
    (G^*(\bm\phi-DGv),z)_\vsp
    & = (\bm\phi-DGv,Gz)_\esp = \ip[D]{D^{-1}(\bm\phi-DGv)}{Gz} \\
    & \leq \norm{D^{-1}(\bm\phi-DGv)}_D \lVert Gz\rVert_D
      = \norm{\bm\phi-DGv}_{D^{-1}} \lVert z\rVert_A.
  \end{align*}
  This proves the result.
\end{remark*}
}

We can minimize the right-hand side of \eqref{eq:functional-estimate}
with respect to $\bm\phi$ to get an estimate of the error. \red{(See
  \cite{ern:hal-01377007} for a treatment of the continuous version of
  the minimization problem via $H(\dvg)$-liftings.)} Denote the
right-hand side of \eqref{eq:functional-estimate} by
\begin{equation*}
  \eta(\bm\phi) = \norm{D\Grad v-\bm\phi}_{D^{-1}}
  + C_P^{-1}\norm{\Div\bm\phi-f}_\vsp.
\end{equation*}
Then the minimization process of $\eta(\bm\phi)$ is described as
follows.

Using the inequality $(a+b)^2\leq(1+\beta)a^2+(1+1/\beta)b^2$ for any
positive $\beta$ one obtains
\begin{equation*}
  \eta^2(\bm\phi) \leq E(\beta,\bm\phi),
\end{equation*}
where
\begin{equation*}
  E(\beta,\bm\phi) \coloneqq (1+\beta)\norm{D\Grad v-\bm\phi}_{D^{-1}}^2 +
  (1+1/\beta)C_P^{-2}\norm{\Div\bm\phi-f}^2_\vsp.
\end{equation*}
An inter-leaved process \cite{MR2817075} can be applied to handle the
minimization of $E$, with respect to $\bm\phi$ and $\beta$
repeatedly. We present the following lemmata that are used for
minimization.

\begin{lemma}\label{lemma:minimizing-phi}
  Let $\beta > 0$ be fixed and let
  \begin{equation*}
    a_1 = 1+\beta, \quad a_2 = (1+1/\beta)C_p^{-2}.
  \end{equation*}
  Then $E(\beta,\bm\phi)$ attains a unique minimum when $\bm\phi$ is
  the solution to
  \begin{equation}\label{eq:minimizing-phi}
    \left(a_1D^{-1}+a_2\Grad \Div\right)\bm\phi=\Grad (a_1v+a_2f).
  \end{equation}
\end{lemma}

\begin{proof}
  We have that
  \begin{equation*}
    E(\beta,\bm\phi) = a_1\norm{\bm\phi-D\Grad v}_{D^{-1}}^2 + a_2\norm{\Div\bm\phi-f}_\vsp^2.
  \end{equation*}
  Since $E$ is convex in $\bm\phi$, $E$ attains minimum if and only if
  the directional derivative of $E$ with respect to $\bm\phi$ is zero
  at some point. For any fixed $\bm\phi$, take a small variation in
  the direction of $\bm\chi$ we get the directional derivative
  \begin{equation*}
    \bm\chi \mapsto \left.\frac{\dif}{\dif\varepsilon}\right\vert_{\varepsilon=0}E(\beta,\bm\phi+\varepsilon\bm\chi),
  \end{equation*}
  where
  \begin{align*}
    & \hspace{1em} \left.\frac{\dif}{\dif\varepsilon}\right\vert_{\varepsilon=0}
      E(\beta,\bm\phi+\varepsilon\bm\chi)
      = \left.\frac{\dif}{\dif\varepsilon}\right\vert_{\varepsilon=0}
      \Big[ a_1\ip[D^{-1}]{\bm\phi+\varepsilon\bm\chi-D\Grad v}{\bm\phi+\varepsilon\bm\chi-D\Grad v} \\
    & \hspace{11em} + a_2\ip[\vsp]{\Div(\bm\phi+\varepsilon\bm\chi)-f}{\Div(\bm\phi+\varepsilon\bm\chi)-f} \Big] \\
    & = \left.\frac{\dif}{\dif\varepsilon}\right\vert_{\varepsilon=0}
      \Big[ a_1\left( \varepsilon^2\ip[D^{-1}]{\bm\chi}{\bm\chi}
      + 2\varepsilon\ip[D^{-1}]{\bm\chi}{\bm\phi-D\Grad v}
      + \ip[D^{-1}]{\bm\phi-D\Grad v}{\bm\phi-D\Grad v}\right) \\
    & \hspace{4em} + a_2\left(\varepsilon^2\ip[\vsp]{\Div\bm\chi}{\Div\bm\chi}
      + 2\varepsilon\ip[\vsp]{\Div\bm\chi}{\Div\bm\phi-f}
      + \ip[\vsp]{\Div\bm\phi-f}{\Div\bm\phi-f}\right) \Big]\\
    & = 2a_1 \ip[D^{-1}]{\bm\chi}{\bm\phi-D\Grad v} +
      2a_2 \ip{\Div\bm\chi}{\Div\bm\phi-f}_\vsp \\
    & = 2a_1 \ip{\bm\chi}{D^{-1}\bm\phi-\Grad v}_{\esp} +
      2a_2 \ip{\bm\chi}{\Grad\Div\bm\phi-\Grad f}_{\esp} \\
    & = 2 \ip[\esp]{\bm\chi}
      {\left(a_1D^{-1}+a_2\Grad\Div\right)\bm\phi-\Grad (a_1v+a_2f)}.
  \end{align*}
  Setting this to be zero we get equation \eqref{eq:minimizing-phi}.
\end{proof}

\begin{lemma}\label{lemma:minimizing-beta}
  Let $\bm\phi$ be fixed such that
  \begin{equation*}
    b_1=\|D\Grad v-\bm\phi\|_{D^{-1}} \;\;\text{and} \quad
    b_2 = C_P^{-1}\|\Div\bm\phi-f\|_\vsp
  \end{equation*}
  are nonzero. Then
  \begin{equation}\label{eq:minimizing-beta}
    \argmin_{\beta}E(\beta,\bm\phi) = \frac{b_2}{b_1}.
  \end{equation}
\end{lemma}

\begin{proof}
  We have that
  \begin{align*}
    E(\beta,\bm\phi) & = (1+\beta)b_1^2+(1+1/\beta)b_2^2 = b_1^2+b_2^2+b_1^2\beta+\frac{b_2^2}{\beta} \\
                     & \geq b_1^2+b_2^2+2\sqrt{b_1^2\beta\cdot\frac{b_2^2}{\beta}} = (b_1+b_2)^2,
  \end{align*}
  where the equal sign holds if and only if
  $\beta = \dfrac{b_2}{b_1}$, giving equation
  \eqref{eq:minimizing-beta}.
\end{proof}

The minimization of $E$ is done by repeatedly minimizing
$E(\beta,\bm\phi)$ with respect to $\bm\phi$ and $\beta$ until $E$
converges. In fact, this process also gives the global minimum of
$\eta(\bm\phi)$. We have the following lemma.

\begin{lemma}
  If $E(\beta,\bm\phi)$ attains a local minimum at
  $(\beta_0,\bm\phi_0)$, then $\eta(\bm\phi)$ attains a global minimum
  at $\bm\phi_0$.
\end{lemma}

\begin{proof}
  Since $\eta(\bm\phi)$ is also a convex function, $\eta(\bm\phi)$
  attains a global minimum if its directional derivative is zero at
  some $\bm\phi$. The directional derivative of $\eta(\bm\phi)$ at
  $\bm\phi_0$ is
  \begin{align*}
    & \left.\frac{\dif}{\dif\varepsilon}\right\vert_{\varepsilon=0}
      \eta(\bm\phi_0+\varepsilon\bm\chi)
      = \left.\frac{\dif}{\dif\varepsilon}\right\vert_{\varepsilon=0}
      \left[ \ip[D^{-1}]{\bm\phi_0+\varepsilon\bm\chi-D\Grad v}{\bm\phi_0+\varepsilon\bm\chi-D\Grad v}^{\half} \right. \\
    & \hspace{13.5em} + \left.C_P^{-1}\ip[\vsp]{\Div(\bm\phi_0+\varepsilon\bm\chi)-f}{\Div(\bm\phi_0+\varepsilon\bm\chi)-f}^{\half} \right] \\
    & \hspace{3em} = \frac{1}{2\norm{DGv-\bm\phi_0}_{D^{-1}}}\left.\frac{\dif}{\dif\varepsilon}\right\vert_{\varepsilon=0}
      \ip[D^{-1}]{\bm\phi_0+\varepsilon\bm\chi-D\Grad v}{\bm\phi_0+\varepsilon\bm\chi-D\Grad v} \\
    & \hspace{4em} + \frac{1}{2C_P\norm{\Div\bm\phi_0-f}_\vsp}\left.\frac{\dif}{\dif\varepsilon}\right\vert_{\varepsilon=0}
      \ip[\vsp]{\Div(\bm\phi_0+\varepsilon\bm\chi)-f}{\Div(\bm\phi_0+\varepsilon\bm\chi)-f}\\
    & \hspace{3em} = \frac{1}{\norm{DGv-\bm\phi_0}_{D^{-1}}} \ip[\esp]{\bm\chi}{D^{-1}\bm\phi_0-\Grad v} +
      \frac{1}{C_P\norm{\Div\bm\phi_0-f}_\vsp} \ip[\esp]{\bm\chi}{\Grad\Div\bm\phi_0-\Grad f}.
  \end{align*}
  Let $b_1 = \norm{DGv-\bm\phi_0}_{D^{-1}}$ and $b_2 = C_P^{-1}\norm{\Div\bm\phi_0-f}_\vsp$, then
  \begin{align}
    \left.\frac{\dif}{\dif\varepsilon}\right\vert_{\varepsilon=0}
    \eta( & \bm\phi_0 +\varepsilon\bm\chi)
      = \frac{1}{b_1} \ip[\esp]{\bm\chi}{D^{-1}\bm\phi_0-\Grad v} +
      \frac{1}{C_P^2b_2} \ip[\esp]{\bm\chi}{\Grad\Div\bm\phi_0-\Grad f} \nonumber\\
    & = \ip[\esp]{\bm\chi}{\left(\frac{1}{b_1}D^{-1}+\frac{1}{C_P^2b_2}\Grad\Div\right)\bm\phi_0
      - \Grad\left(\frac{1}{b_1}v+\frac{1}{C_P^2b_2}f\right)}. \label{eq:derivative-eta}
  \end{align}

  Since $\beta_0$ minimizes $E(\beta,\bm\phi_0)$, by lemma
  \ref{lemma:minimizing-beta}, it must be that
  $\beta_0 = \dfrac{b_2}{b_1}$. By lemma \ref{lemma:minimizing-phi}
  and the fact that $\bm\phi_0$ minimizes $E(\beta_0,\bm\phi)$, we
  have
  \begin{equation}\label{eq:minimizing-phi2}
    \left(a_1D^{-1}+a_2\Grad\Div\right)\bm\phi_0 = \Grad(a_1v+a_2f),
  \end{equation}
  where $a_1 = 1+\beta_0 = \dfrac{b_1+b_2}{b_1}$, and
  $a_2 = (1+1/\beta_0)C_P^{-2} = \dfrac{b_1+b_2}{C_P^2b_2}$. Thus
  \eqref{eq:minimizing-phi2} becomes
  \begin{equation*}
    \left(\frac{1}{b_1}D^{-1}+\frac{1}{C_P^2b_2}\Grad\Div\right)\bm\phi_0
    = \Grad\left(\frac{1}{b_1}v+\frac{1}{C_P^2b_2}f\right).
  \end{equation*}
  Comparing this with \eqref{eq:derivative-eta}, we get
  $\displaystyle\left.\frac{\dif}{\dif\varepsilon}\right\vert_{\varepsilon=0}
  \eta(\bm\phi_0+\varepsilon\bm\chi) = 0$, which completes the proof.
\end{proof}

When $\esp$ is of very high dimension, for example in our application,
where $\esp$ is the edge space of a large graph, the minimization
becomes time consuming. We handle this by varying $\bm\phi$ over only
a subspace of $\esp$ (say, the subspaces $W_H$ defined in section
\ref{sec:coarsening}). It is not hard to show that the above
minimization results also hold when $\bm\phi$ is chosen from a
subspace of $\esp$.

\section{Coarsening through aggregation}
\label{sec:coarsening}
In this section we introduce the constructions of coarse subspaces of
the vertex space $\vsp$ and edge space $\esp$ in which approximations
of the Laplacian are considered. In what follows we omit the
subscripts on the inner products as it will always be clear from the
context in which spaces the inner products are taken. An aggregation
on a graph $\cG = (\cV, \cE)$ is a splitting of the vertex set $\cV$
into non-overlapping subsets, each of which is connected in $\cG$ and
called an aggregate:
\begin{equation*}
  \cV = \bigcup_{k=1}^{n_c}\cV_{\cA_k} = \bigcup_{\cA}\cV_{\cA}.
\end{equation*}
Here an aggregate is denoted by $\cA$ and $\cV_{\cA}$ is its set of
vertices. This splitting naturally splits the set of edges into
interior and interface edges. For each aggregate $\cA$, we denote by
$\cE_{\cA}$ the set of interior edges of $\cA$.
\begin{equation*}
  \cE_{\cA} = \{(i,j)\in\cE \mid i, j \in \cV_{\cA}\}.
\end{equation*}
For any two aggregates $\cA$ and $\cA'$, denote by $\cI_{\cA\cA'}$ the
set of interface edges between them, that is, the set of edges which
connect $\cA$ and $\cA'\colon$
\begin{equation*}
  \cI_{\cA\cA'} = \{(i,j)\in\cE \mid i\in\cV_{\cA}, j\in\cV_{\cA'}\}.
\end{equation*}
Note edges in the graph are undirected and $(i,j)$ and $(j,i)$ are
considered to be the same in $\cE$. We denote by $\Gamma$ the set of
all interfaces in the aggregation.

An aggregation of $\cG$ is completely characterized by the subspace of
$\vsp$ consisting of vectors taking one and the same value at all
vertices in an aggregate. We denote this subspace by $\vsp_H$ as
follows.
\begin{equation*}
  \vsp_H = \{u\in \vsp \mid u_i = u_j \text{ whenever } i,j\in \cV_{\cA}
  \text{ for some } \cA\}.
\end{equation*}
This admits a two-level hierarchy. The solution of the variational
problem $Au=f$ can be approximated by (for instance) some $u_H$ in
$\vsp_H$ by solving:
\begin{equation}\label{eq:uH}
  (Au_H, v_H) = (f, v_H), \quad \mforall v_H\in \vsp_H.
\end{equation}

The results in the previous section can be used to establish a
posteriori error estimates for the approximate solution $u_H$. In
order to get computable error estimates we also need coarse subspaces
of the edge space $\esp$, from which the $\bm\phi$ in
\eqref{lem:Repin} will be chosen. \red{We point out that in the
  context of a posteriori estimates for finite element methods,
  $\bm\phi$ is sometimes chosen to be the equilibrated flux to make
  the second term $\lVert \Grad^*\bm\phi -f \rVert$ small or zero
  \cite{MR3335498}}. Below we present two constructions of coarse edge
spaces in an aggregation, which are both Raviart-Thomas-like.

\subsection{Coarse edge space via saddle point problem}
We introduce for each interface $\cI=\cI_{\cA\cA{'}}\in\Gamma$ the
vector
$\bm\sigma_\cI = \varepsilon_{\cA\cA'}\mathbf{Q}_\cI\Grad\vone_\cA$,
where $\varepsilon_{\cA\cA'} = \pm 1$ is predetermined for the
interface and $\mathbf{Q}_\cI$ is the orthogonal projection onto
$\cI$. For a given $\bm\psi\in\esp$ define the averaging operator
\begin{equation*}
  \avedg{\bm\psi}{\cI}=\frac{(\bm\psi,\bm\sigma_\cI)}{\norm{\bm\sigma_\cI}^2}.
\end{equation*}
In the context of standard finite element methods, this can be viewed
as an analogue of Raviart-Thomas degree of freedom on a coarser grid,
namely, an analogue of averaging the normal trace of a vector
field. The basis for the coarse edge space $\esp_H$ is constructed via
solving the following local saddle point problem: find
$\bm\varphi_{\cI,\cA}\in\esp_{\cE_\cA}$ and $u_{\cI,\cA}\in \vsp_\cA$,
$(u_{\cI,\cA},\vone)=0$, such that
\begin{IEEEeqnarray*}{rCl'l}
  \cB(\bm\varphi_{\cI,\cA},\bm\psi)+(u_{\cI,\cA},\Div\bm\psi)
  & = & -\cB(\bm\sigma_{\cI},\bm\psi),
  & \mforall\bm\psi\in\esp_{\cE_\cA};\label{local saddle--point problem1} \\
  (\Div\bm\varphi_{\cI,\cA},v) & = & -(\Div\bm\sigma_\cI,v),
  & \mforall v\in \vsp_\cA,(v,\vone)=0.\label{local saddle--point problem2}
\end{IEEEeqnarray*}
An example of a bilinear form
$\cB(\cdot,\cdot)\colon\esp\times\esp\to\Reals$ is the $\ell^2$-inner
product on $\esp$, i.e.,
$\cB(\bm\varphi,\bm\psi) = \sum_{e\in\cE}\bm\varphi_e\bm\psi_e$. One
can also use a different bilinear form for this. The above saddle
problem is the Lagrange multiplier formulation of the following
constraint minimization problem \cite{franco-inf-sup}.
\begin{align*}
  & \bm\varphi_{\cI,\cA} = \argmin_{\bm\psi\in\esp_{\cE_\cA}}\{\cB(\bm\psi+\bm\sigma_\cI,\bm\psi+\bm\sigma_\cI)\}, \\
  & \text{subject to}\quad(\Div\bm\psi,v)=-(\Div\bm\sigma_\cI,v), \ \mforall v\in \vsp_\cA,(v,\vone)=0.
\end{align*}
We solve an analogous problem on $\cA'$ and define the basis function
$\bm\varphi_\cI$ as
\begin{equation*}
  (\bm\varphi_\cI)_e=
  \begin{cases}
    (\bm\varphi_{\cI,\cA})_e,    & \text{if }e\in\cE_\cA;   \\
    (\bm\sigma_\cI)_e,          & \text{if }e\in\cI;      \\
    (\bm\varphi_{\cI,\cA'})_e,   & \text{if }e\in\cE_{\cA'}; \\
    0,                         & \text{for other edges}.
  \end{cases}
\end{equation*}

Define $\pi_H\colon \esp \to \esp_H$ via the canonical interpolation
from the space
$\operatorname{span}\{\bm\varphi_{\cI}\}_{\cI\in \Gamma}$. For a given
$\bm\psi\in\esp$, set
\begin{equation}\label{eq-definition-of-pi-H}
  \pi_H\bm\psi=\sum_{\cI\in\Gamma}\avedg{\bm\psi}{\cI}\bm\varphi_\cI.
\end{equation}
We then have the following theorem.

\begin{theorem}[Theorem~5.3, \cite{2013VassilevskiP_ZikatanovL-aa}]
  \label{thm:commuting-piH}
  Let $Q_H$ be the $\ell^2$-based projection on the space $\vsp_H$,
  that is, averaging on every aggregate. Then for all $\bm\psi\in\esp$
  and all $v\in\vsp$ we have
  \begin{equation*}
    (\Div\pi_H\bm\psi,v)=(Q_H\Div\bm\psi,v).
  \end{equation*}
\end{theorem}

The theorem says that the following diagram commutes.
\begin{equation*}
  \begin{CD}
    \color{red}{\esp} @> {\Div} >> {\vsp}     \\
    @ V{\pi_H} VV @ VV {Q_H} V   \\
    {\esp_H} @>> {\Div} > {\vsp_H} \\
  \end{CD}.
\end{equation*}

\subsection{Coarse edge space via spanning tree}
Another way of constructing $\esp_H$ and $\pi_H$ such that the above
diagram commutes is via the spanning trees of the aggregates. For each
interface edge $e=(i,j)$ between aggregates $\cA$ and $\cB$
($i\in\cA, j\in\cB$), fix a spanning tree of $\cA$ rooted at $i$ (note
$\cA$ can be viewed a subgraph of $\cG$), see figure
\ref{fig:spanning-tree}. For an edge $e'$ in the tree and the
corresponding child node $i'$, denote by $m_{e'}$ the size of the
subtree rooted at $i'$. Define $\bm\varphi^e_\cA\in\esp_{\cE_\cA}$ by
\begin{equation*}
  \left(\bm\varphi^e_\cA\right)_{e'} = 
  \begin{cases}
    \pm\dfrac{m_{e'}}{\left\lvert\cV_\cA\right\rvert},
    & \text{if } e' \text{ is an edge of the tree;} \\
    0, & \text{otherwise.}
  \end{cases}
\end{equation*}
Here the ``$\pm$'' takes ``$+$'' sign if the vertex $i'$ is the head
of $e'$ according to the predetermined orientation on the edges of the
graph, and ``$-$'' otherwise.

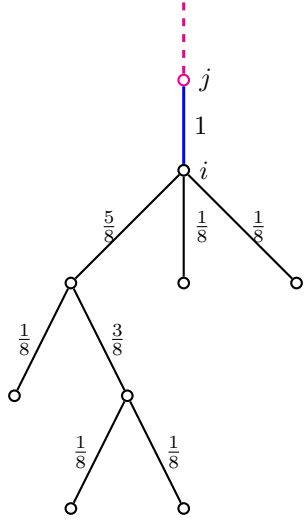
\begin{figure}
  \centering
  \begin{tikzpicture}[
    treenode/.style={circle,draw=black,thick,minimum size=4pt,
      inner sep=0pt}]
    
    \node [treenode,draw=magenta,label={right:$j$}] {}
    child [grow=down,level distance=1.2cm,
    every child/.style={level distance=1.5cm,
      edge from parent/.style={draw=black,thick}}] {
      node [treenode,label={right:$i$}] {}
      child {
        node [treenode] {}
        child {
          node [treenode] {}
          edge from parent node [left] {$\frac{1}{8}$}
        }
        child {
          node [treenode] {}
          child {
            node [treenode] {}
            edge from parent node [left] {$\frac{1}{8}$}
          }
          child {
            node [treenode] {}
            edge from parent node [right] {$\frac{1}{8}$}
          }
          edge from parent node [right] {$\frac{3}{8}$}
        }
        edge from parent node [left] {$\frac{5}{8}$}
      }
      child {
        node [treenode] {}
        edge from parent node [right] {$\frac{1}{8}$}
      }
      child {
        node [treenode] {}
        edge from parent node [right] {$\frac{1}{8}$}
      }
      edge from parent [draw=blue,line width=0.4mm] node [right] {$1$}
    }
    child [grow=up,level distance=1.1cm] {
      edge from parent [draw=magenta,dashed,line width=0.4mm]
    };
  \end{tikzpicture}
  \caption{Construction of $\bm\varphi^e$ via spanning tree. Aggregate
    $\cA$ is in black, $\cB$ is in magenta, and the blue edge is the
    interface edge $e=(i,j)$.}
  \label{fig:spanning-tree}
\end{figure}

Similarly we define $\bm\varphi^e_\cB\in\esp_{\cE_\cA}$ via a spanning
tree of $\cB$ rooted at $j$. Then let $\bm\varphi^e\in\esp$ be such
that
\begin{equation*}
  (\bm\varphi^e)_{e'} =
  \begin{cases}
    \left(\bm\varphi^e_\cA\right)_{e'},  & \text{if } e' \in \cE_\cA; \\
    1,                                  & \text{if } e' = e; \\
    -\left(\bm\varphi^e_\cB\right)_{e'}, & \text{if } e' \in \cE_\cB; \\
    0,                                  & \text{otherwise.}
  \end{cases}
\end{equation*}
We now define $\pi_H\colon\esp\to\esp_H$ in the following way.
\begin{equation*}
  \pi_H\eone_e =
  \begin{cases}
    \bm\varphi^e, & \text{if } e \text{ is an interface edge;} \\
    \bm0,         & \text{otherwise.}
  \end{cases}
\end{equation*}
Then theorem \ref{thm:commuting-piH} holds for the $\pi_H$ as defined
above as well.

\begin{proof}
  Note that we want
  \begin{equation}\label{eq:piH-requirement}
    (\Div\pi_H\bm\psi,v)=(Q_H\Div\bm\psi,v),
    \quad \mforall \bm\psi\in\esp, \ \mforall v\in\vsp.
  \end{equation}
  
  For any edge $e$ interior to an aggregate and for $\eone_e\in\esp$,
  we have $Q_H\Grad^*\eone_e = \bm0$. On the other hand,
  $\pi_H\eone_e = \bm0$ by definition. In this case
  \eqref{eq:piH-requirement} holds.

  For an interface edge $e=(i,j)$, $i\in\cA$, $j\in\cB$, to show
  \eqref{eq:piH-requirement} we observe the left-hand side is
  $(\pi_H\eone_{e}, \Grad v)$ and the right-hand side is
  \begin{equation}\label{eq:piH-condition}
    (\eone_{e}, \Grad Q_Hv)
    = \frac{1}{\lvert\cV_\cA\rvert}\sum_{k\in\cV_\cA}v_k -
    \frac{1}{\lvert\cV_{\cA'}\rvert}\sum_{l\in\cV_{\cB}}v_l.
  \end{equation}
  For $v = \vone_k$ where $k$ is any vertex not in $\cA$ or $\cB$,
  \eqref{eq:piH-condition} takes value 0. By definition $\pi_H\eone_e$
  vanishes on edges that are neither interior to $\cA$ or $\cB$ nor in
  the interface $\cI_{\cA\cB}$, so
  $(\pi_H\eone_{e}, \Grad\vone_k) = 0$. The left-hand side is equal to
  the right-hand side. For $v = \vone_k$ where $k$ is any vertex in
  $\cA$, \eqref{eq:piH-condition} becomes
  $\frac{1}{\lvert\cV_\cA\rvert}$. Thus we need to verify
  \begin{equation*}
    (\pi_H\eone_e, \Grad\vone_k) = \frac{1}{\lvert\cV_\cA\rvert}
    \quad\text{or}\quad
    (\bm\varphi^e, \Grad\vone_k) = \frac{1}{\lvert\cV_\cA\rvert}.
  \end{equation*}
  This is immediate from the construction of $\bm\varphi^e$. The same
  argument applies if $k$ is in $\cB$.
\end{proof}

A basis for $\esp_H$ is then constructed in the similar fashion as in
the previous subsection, specifically, for $\cI = \cI_{\cA\cB}$,
\begin{equation*}
  \bm\varphi_\cI \coloneqq \sum_{e\in\cI}\delta_e\bm\varphi^e,
  \quad \text{where} \quad e=(i,j), \quad
  \delta_e =
  \begin{cases}
    1,  & \text{ if } i \in \cA; \\
    -1, & \text{ if } i \in \cB.
  \end{cases}
\end{equation*}

\section{Applications to aggregation in graphs}
We apply the error estimator in lemma \ref{lem:Repin} to the graph
Laplacian problem \eqref{eq:problem}, where $\vsp$ and $\esp$ are just
the vertex and edge spaces of a graph. The standard graph Laplacian
$A$ takes the form $A = \Div\id_{\esp}\Grad$, and the norms
$\norm{\cdot}_\vsp$ and $\norm{\cdot}_{\esp}$ are $\ell^2$-norms. We
get the following estimate for the error of the approximate solution
$u_H$.
\begin{equation}\label{eq:estimator}
  \norm{u-u_H}_A \leq \norm{\Grad u_H - \bm\phi} +
  C_P^{-1}\norm{\Div\bm\phi - f}
\end{equation}
Note the graph Laplacian $A$ is only positive semidefinite. However if
we replace $\vsp$ with the hyperplane in the vertex space that is
orthogonal to the first eigenvector $[1\ 1\ \cdots\ 1]^T$ of $A$ then
$A$ becomes positive definite. In this case $C_P$ is the smallest
positive (second) eigenvalue of $A$.

We use the error estimate \eqref{eq:estimator} to devise an algorithm
that generates aggregations adaptive to not only $A$, but also the
right-hand side $f$. For given $f$ and any aggregation we can minimize
the right-hand side of \eqref{eq:estimator} over a coarse edge space
$\esp_H$. We denote this minimum by
\begin{equation*}
  \eta(f,\cA) = \inf_{\bm\phi\in\esp_H} \norm{\Grad u_H - \bm\phi}
  + C_P^{-1}\norm{\Div\bm\phi - f}.
\end{equation*}
We then devise a reshaping algorithm that finds a new aggregation,
aiming to reduce the this estimator. To do this we first localize our
estimator. Note
\begin{equation*}
  \norm{\Grad u_H - \bm\phi} + C_P^{-1}\norm{\Div\bm\phi - f}
  \leq 2^{\frac{1}{2}}\tilde\eta^{\frac{1}{2}},
\end{equation*}
where $\tilde\eta$ is
\begin{IEEEeqnarray*}{rCl}
  \tilde\eta
  & \coloneqq & \norm{\Grad u_H-\bm\phi}^2+C_P^{-2}\norm{\Div\bm\phi-f}^2 \\
  & = & \sum_{\cI}\sum_{e\in\cI}\norm{\Grad u_H-\bm\phi}_e^2 + \sum_{\cA}\bigg[\sum_{e\in\cE_{\cA}}\norm{\Grad u_H-\bm\phi}_e^2 \\
  &   & \hspace{1em}+C_P^{-2}\norm{\Div\bm\phi-f}_{\cV_{\cA}}^2\bigg] \\
  & = & \sum_{\cA} \bigg[\frac{1}{2}\sum_{\cA'}\sum_{e\in\cI_{\cA\cA'}}\norm{\Grad u_H-\bm\phi}_e^2+\sum_{e\in\cE_{\cA}}\norm{\Grad u_H-\bm\phi}_e^2 \\
  &   & \hspace{1em}+C_P^{-2}\norm{\Div\bm\phi-f}_{\cV_{\cA}}^2\bigg] \\
  & \eqqcolon & \sum_{\cA}\tilde\eta_{\cA}.
\end{IEEEeqnarray*}
The reshaping algorithm \ref{alg:reshaping} then iteratively splits
the aggregates on which $\tilde\eta_{\cA}$ is large, until some
stopping criterion is met. The stop criterion can be, e.g., the number
of aggregates $n_c$ becomes larger than a given threshold $N$, or the
error estimate $\eta(f,\cA)$ drops below some preset value.

\begin{algorithm}
  \caption{Reshaping of aggregation}\label{alg:reshaping}
  \begin{algorithmic}[1]
    \State \red{Suppose a graph $\cG$ and an aggregation
      $\{\cA_k\}_{k=1}^{n_c}$ are given.}
    
    \State Compute the approximate solution $u_H$. Then find $\bm\phi$
    that minimizes $\eta(f,\cA)$\label{alg:init}.
    
    \State Split all $\cA_k$ for which\label{alg:split}
    \begin{equation*}
      \tilde\eta_{\cA_k} > \frac{\sum_{i=1}^{n_c}\tilde\eta_{\cA_i}}{n_c}.
    \end{equation*}
    
    \State Check if the stopping criterion is met, if not go to step
    \ref{alg:init}.
  \end{algorithmic}
\end{algorithm}
To mark aggregates for refinement (splitting) we have used an analogue
of the equidistribution marking well known in adaptive finite element
methods. Other strategies for marking, such as maximum and D\"oerfler
type marking \cite{2009NochettoSiebert, 1996DoerflerW-aa} are subject
of future research.  All aggregates marked for refinement are split
using the hierarchy already available via matching described in the
next section, since this can be done at virtually zero computational
cost.

\section{Numerical examples: matching, un-matching and reshaping}
In this section we perform several numerical experiments on the error
estimates and reshaping algorithm proposed above.

\subsection{Experiment I: guided reshaping starting from a coarse
  aggregation}
We test the reshaping algorithm using matching on graphs. By matching
we mean an algorithm that aggregates a graph by grouping two vertices
together at a time \cite{MR1639073}. To be exact, the matching
algorithm works as follows.
\begin{enumerate}
\item Choose a vertex of smaller degree and group it with one of its
  unmatched neighbors, if such neighbor exists. The degree of a vertex
  is defined to be the number of edges attached to this vertex.
\item Repeat this until there are only isolated vertices which have no
  unmatched neighbors. Then group each isolated vertex with a neighbor
  with which it has the most connections.
\end{enumerate}
The matching algorithm can be performed repeatedly to generate a
hierarchy of coarse aggregates.

Figure \ref{fig:reshaping-scheme} then shows a guideline of the
experiment. We start from $V$, the finest (original) vertex space, and
iteratively apply the matching algorithm. Let $\vsp_{\widetilde{H}}$
be the subspace of $\vsp$ obtained after $k_0$ iterations of matching,
and $\vsp_J$ be after $k_1$ iterations, where $k_0<k_1$ so that
$\vsp_J$ is coarser than $\vsp_{\widetilde{H}}$. We solve for
$u_{\widetilde{H}}$ on $\vsp_{\widetilde{H}}$ and get the error
$e_{\widetilde{H}} = \|u-u_{\widetilde{H}}\|_A$, with the right-hand
side $f$ of \eqref{eq:problem} being a ``smooth'' vector (smoothed by
Gauss-Seidel iterations). The reshaping algorithm is then performed on
$\vsp_J$ iteratively, where the chosen aggregates are split in the
same way they were grouped during matching. The newly generated
aggregations are represented by nodes on the oblique line in figure
\ref{fig:reshaping-scheme}. We do this iteratively until the error
$e_H$ on $\vsp_H$ becomes smaller than $e_{\widetilde{H}}$. We then
compare $\vsp_{\widetilde{H}}$ and $\vsp_H$ to see if the algorithm
reduces the number of aggregates when achieving the same
error. \red{All graphs in the following examples are taken from the
  SuiteSparse Matrix Collection.}

\begin{figure}
  \centering
  \begin{tikzpicture}[scale=.8,auto,swap,
    vertex/.style={circle,draw,thick,minimum size=4pt,inner sep=0pt},
    edge/.style={draw,thick,-}, dashed
    edge/.style={draw,thick,dashed}]
    \foreach \pos / \number / \name / \fr in {(0,5.5)/0/V/2,
      (0,4.5)/1/V_1/3, (0,3.5)/2/V_2/4, (0,2)/3/V_{\widetilde{H}}/6,
      (0,0)/4/V_J/7, (.6,1)/5//10}
    \node[vertex,label={left:$\name$}] (\number) at \pos {};
    \foreach \pos / \number / \name / \fr in {(1.2,2)/6/V_H/11}
    \node[vertex,label={right:$\name$}] (\number) at \pos {};
    \foreach \pos / \note in {(2.5,2.5) / \dim V_H < \dim
      V_{\widetilde{H}}, (2.5,1.6) / \lVert e_H\rVert_A = \lVert
      e_{\widetilde{H}}\rVert_A}
    \node[label={right:$\note$}] at \pos {};
    \foreach \source / \dest / \fr in {0/1/3, 1/2/4, 4/5/10}
    \path[edge] (\source) -- (\dest);
    \foreach \source / \dest / \fr in {2/3/6, 3/4/7, 5/6/11}
    \path[dashed edge] (\source) -- (\dest);
  \end{tikzpicture}
  \caption{Guided reshaping scheme for generating aggregations on
    graphs}
  \label{fig:reshaping-scheme}
\end{figure}
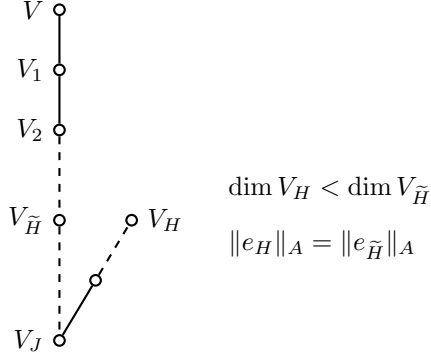

\subsubsection*{\textbf{Example 1}}
The first example is the graph barth5. barth5 has 15606 vertices and
45878 edges. The second eigenvalue of its standard graph Laplacian
matrix is $C_P = 0.02776$, and on average each vertex has 5.9
neighbors.  The results are shown in table
\ref{tab:reshaping-barth5}. We see that in achieving the same error,
reshaping significantly reduces the number of aggregates compared to
matching. Note although the estimation via spanning tree does not give
as low efficiency, it results in more reduction of $n_c$ in many
cases. Figure \ref{fig:reshaping-barth5} plots the aggregations
obtained using matching with $k_0=6$ and reshaping started from
$k_1=7$. Note that with reshaping (figure \ref{fig:reshaping2}), fewer
aggregates are generated and aggregates close to the ``boundary'' are
more likely to be grouped together.

\begin{table}
  \centering
  \begin{tabular}{l|r|c|@{}c@{}|r|c|@{}c@{}}
    & \multicolumn{3}{c|}{$\esp_H$ via saddle point problem}
    & \multicolumn{3}{c}{$\esp_H$ via spanning trees} \\\hline
    Aggregation & $n_c$ & $\eff$ & $\norm{u_H-u}_A/\norm{u}_A$
    & $n_c$ & $\eff$ & $\norm{u_H-u}_A/\norm{u}_A$ \\\hline
    $k_0=3$                & $1798$ & $1.9465$ & $0.8456$   & $1798$ & $1.4298$ & $0.8456$ \\
    $k_1=7$ then reshape   & $742$  & $1.1667$ & $0.8223$   & $699$  & $1.6414$ & $0.8224$ \\\hline
    $k_0=4$                & $837$  & $1.7117$ & $0.9020$   & $837$  & $1.4897$ & $0.9020$ \\
    $k_1=7$ then reshape   & $354$  & $1.1532$ & $0.8939$   & $312$  & $1.5308$ & $0.9000$ \\\hline
    $k_0=5$                & $395$  & $1.5754$ & $0.9335$   & $395$  & $1.4731$ & $0.9335$ \\
    $k_1=7$ then reshape   & $250$  & $1.1587$ & $0.9195$   & $216$  & $1.4928$ & $0.9262$ \\\hline
    $k_0=6$                & $190$  & $1.4575$ & $0.9586$   & $190$  & $1.4091$ & $0.9586$ \\
    $k_1=7$ then reshape   & $143$  & $1.1468$ & $0.9499$   & $155$  & $1.4574$ & $0.9446$ \\\hline
  \end{tabular}
  \caption{Comparison of direct matching and reshaping on the graph
    barth5, at different levels of coarsening. We show results
    obtained for both constructions of $\esp_H$ introduced in section
    \ref{sec:coarsening}. Here
    $\eff \coloneqq \eta/\lVert u_h-u\rVert_A$ represents the
    efficiency of the operator.}
  \label{tab:reshaping-barth5}
\end{table}

\begin{figure}
  \centering
  \begin{subfigure}{0.49\textwidth}
    \centering
    \includegraphics[width=\linewidth]{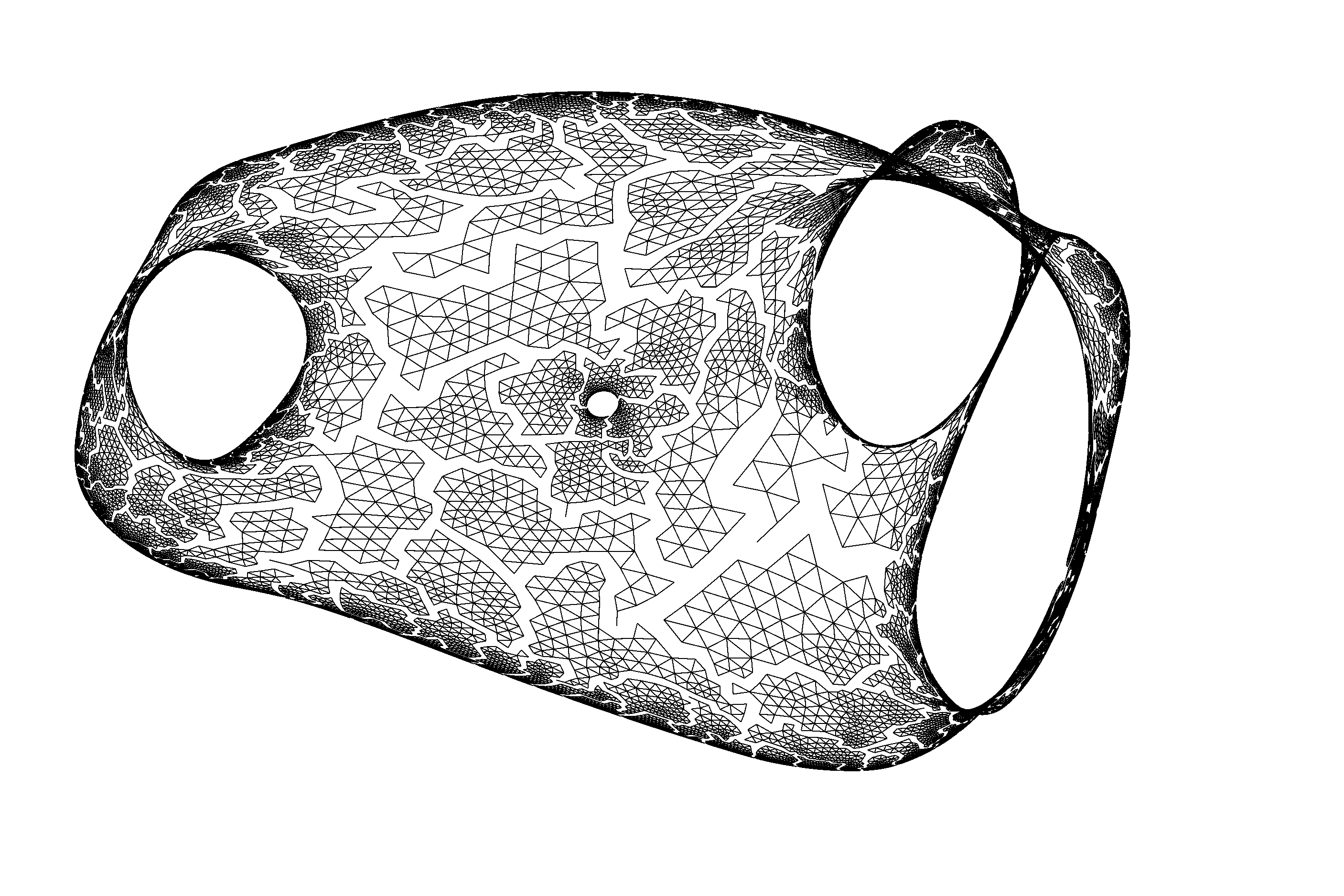}
    \caption{$k_0=6$, $n_c=190$}
  \end{subfigure}
  \begin{subfigure}{0.49\textwidth}
    \centering
    \includegraphics[width=\linewidth]{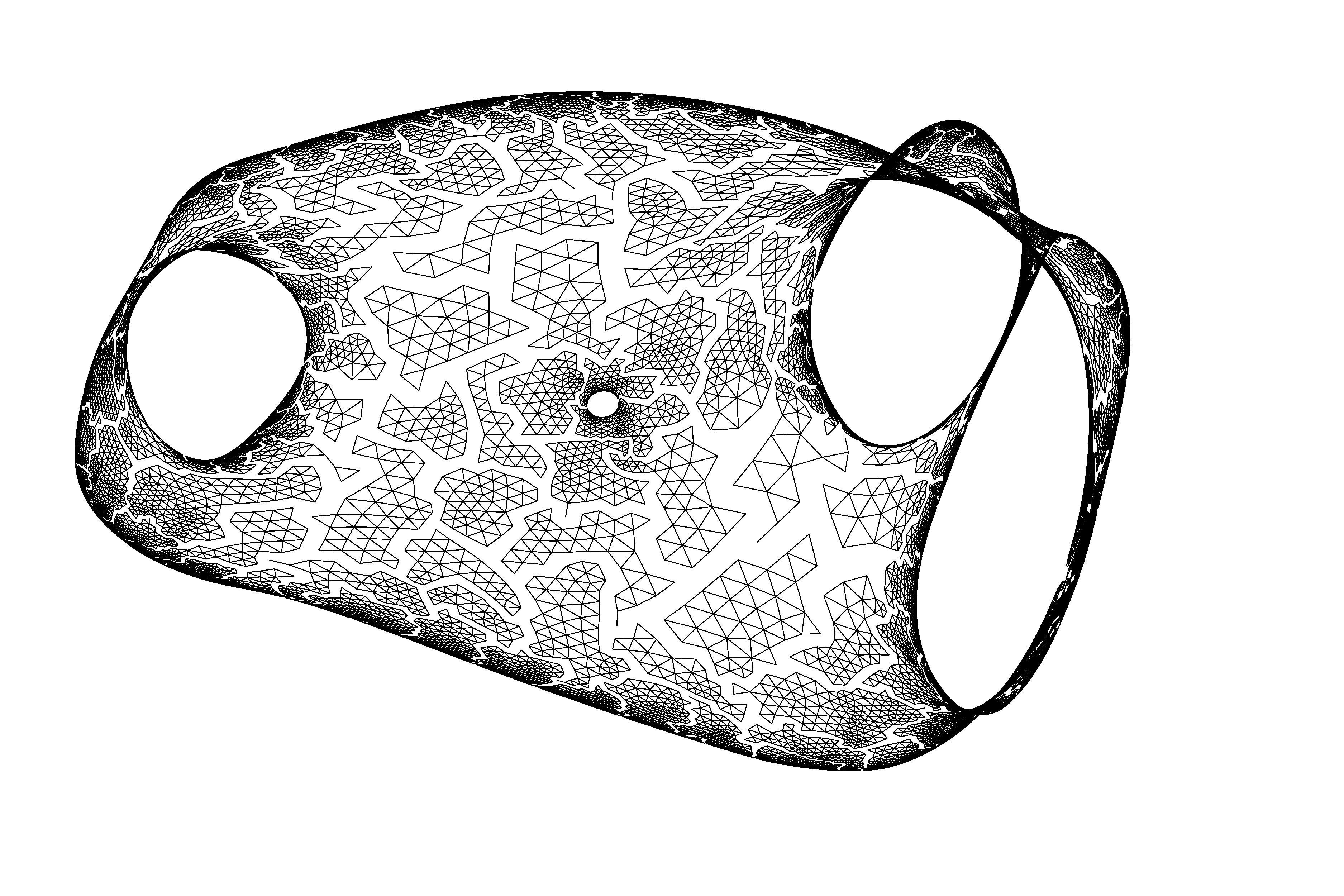}
    \caption{$k_1=7$ with reshaping, $n_c=129$}
    \label{fig:reshaping2}
  \end{subfigure}
  \caption{Plots of aggregations constructed with and without
    reshaping on the graph barth5}
  \label{fig:reshaping-barth5}
\end{figure}

\subsubsection*{\textbf{Example 2}}
The second graph we test is the graph power. power is a smaller graph
of 4941 vertices and 6594 edges. The second eigenvalue of its
Laplacian matrix is $C_P=0.02755$. The vertices have an average degree
of 2.7. Table \ref{tab:reshaping-power} and figure
\ref{fig:reshaping-power} show the numerical results. Same as in
example 1, the reshaping algorithm is able to achieve smaller
approximation error with fewer aggregates. Again the method using
$\esp_H$ via spanning trees is more selective in choosing the
aggregates to split and thus results in fewer aggregates at the same
level.

\begin{table}
  \centering
  \begin{tabular}{l|r|c|@{}c@{}|r|c|@{}c@{}}
    & \multicolumn{3}{c|}{$\esp_H$ via saddle point problem}
    & \multicolumn{3}{c}{$\esp_H$ via spanning trees} \\\hline
    Aggregation & $n_c$ & $\eff$ & $\norm{u_H-u}_A/\norm{u}_A$
    & $n_c$ & $\eff$ & $\norm{u_H-u}_A/\norm{u}_A$ \\\hline
    $k_0=3$                & $441$ & $1.1876$ & $0.8373$   & $441$ & $1.2571$ & $0.8373$ \\
    $k_1=7$ then reshape   & $121$ & $1.2703$ & $0.8346$   & $150$ & $1.8936$ & $0.7855$ \\\hline
    $k_0=4$                & $188$ & $1.2797$ & $0.8798$   & $188$ & $1.4082$ & $0.8798$ \\
    $k_1=7$ then reshape   & $121$ & $1.2703$ & $0.8346$   & $75$  & $1.8390$ & $0.8786$ \\\hline
    $k_0=5$                & $86$  & $1.3407$ & $0.9225$   & $86$  & $1.5524$ & $0.9225$ \\
    $k_1=7$ then reshape   & $62$  & $1.3779$ & $0.9107$   & $55$  & $1.8603$ & $0.9018$ \\\hline
    $k_0=6$                & $40$  & $1.4681$ & $0.9495$   & $40$  & $1.7373$ & $0.9495$ \\
    $k_1=7$ then reshape   & $39$  & $1.3945$ & $0.9398$   & $38$  & $1.8234$ & $0.9410$ \\\hline
  \end{tabular}
  \caption{Comparison of direct matching and reshaping on the graph
    power, at different levels of coarsening. We show results obtained
    for both constructions of $\esp_H$ introduced in section
    \ref{sec:coarsening}.}
  \label{tab:reshaping-power}
\end{table}

\begin{figure}
  \centering
  \begin{subfigure}{0.50\textwidth}
    \includegraphics*[width=\linewidth]{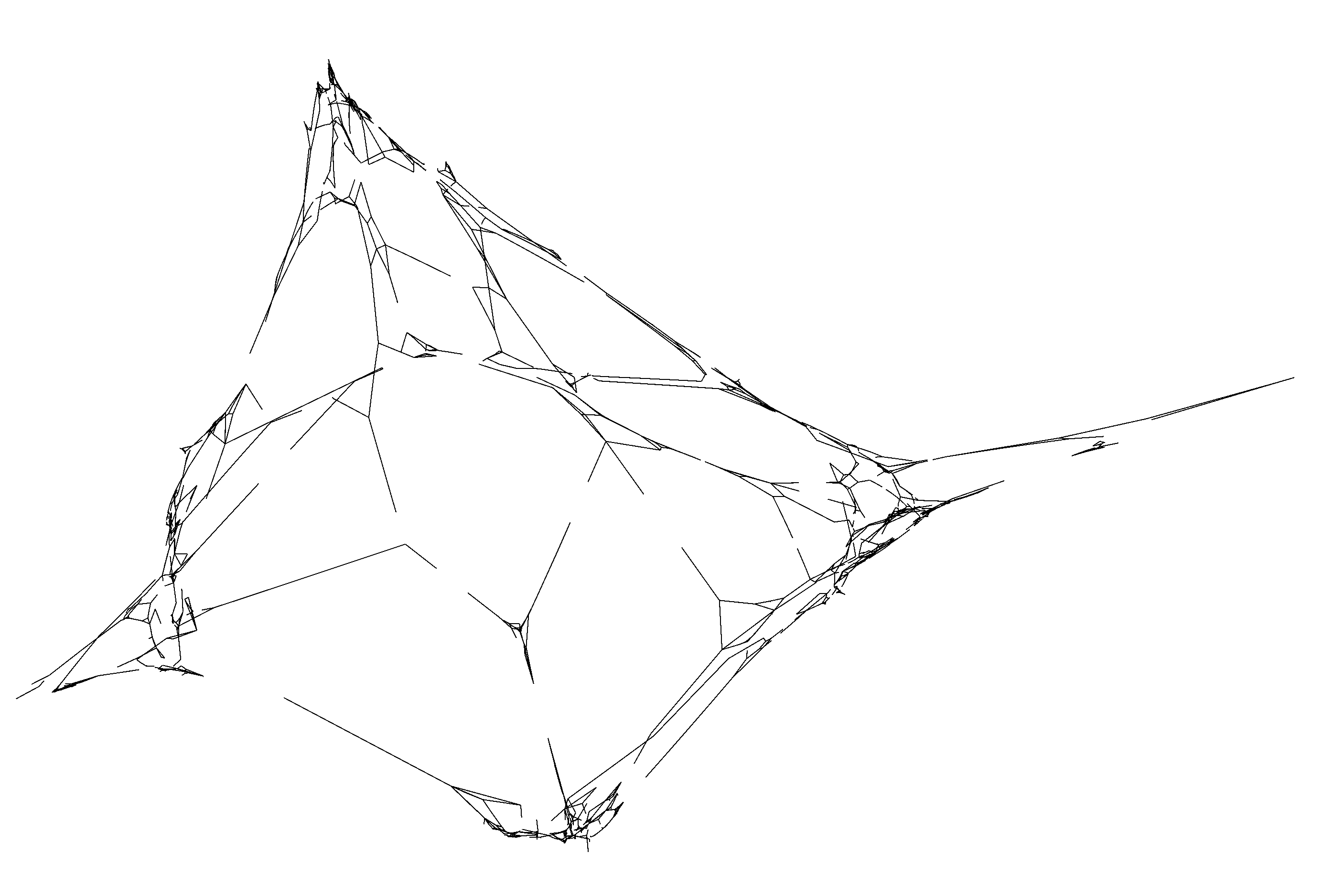}
    \caption{$k_0=5$, $n_c=86$}
  \end{subfigure}~
  \begin{subfigure}{0.50\textwidth}
    \includegraphics*[width=\linewidth]{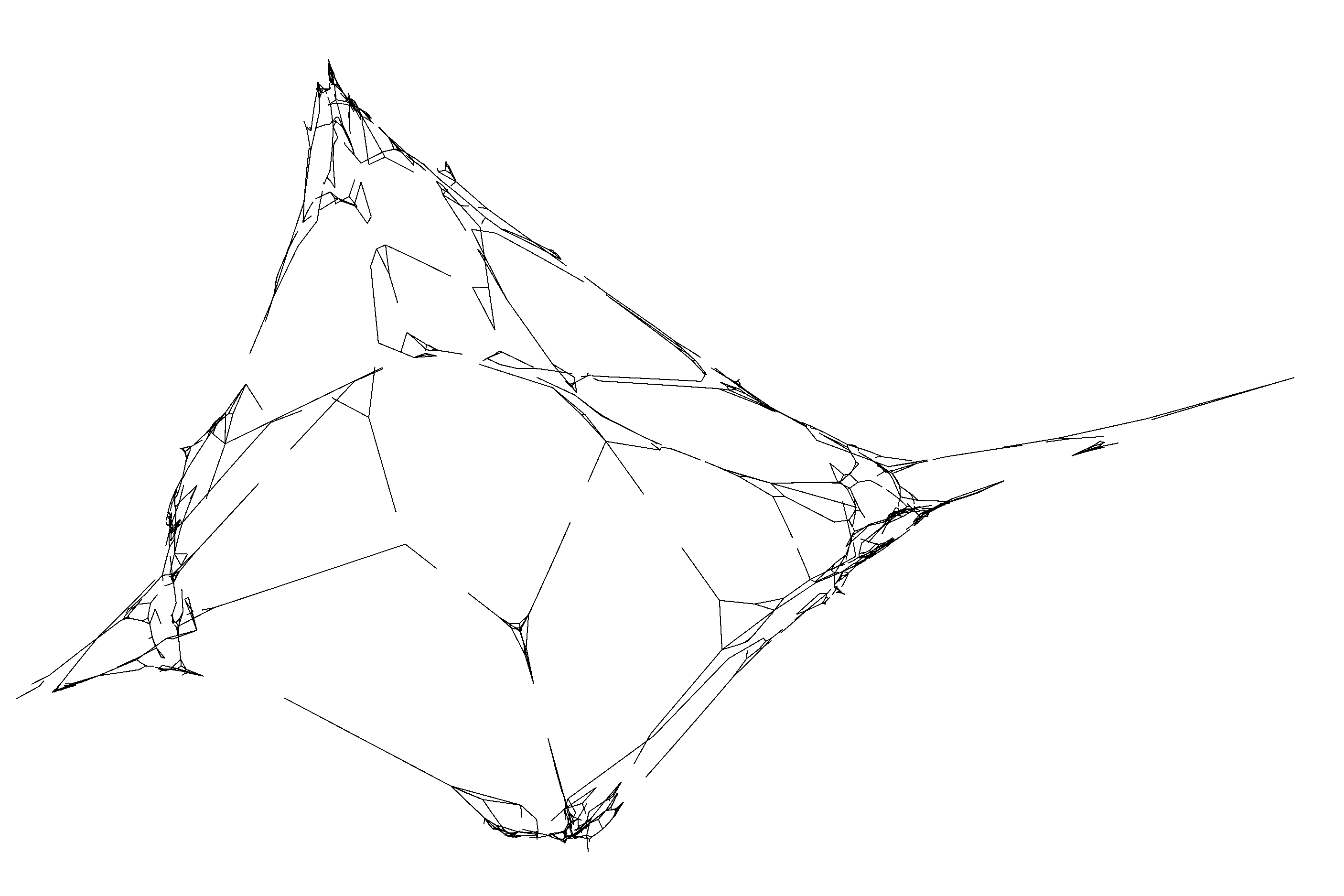}
    \caption{$k_1=7$ with reshaping, $n_c=55$}
  \end{subfigure}
  \caption{Plots of aggregations constructed with and without
    reshaping on the graph power}
  \label{fig:reshaping-power}
\end{figure}

\subsubsection*{\textbf{Example 3}}
Table \ref{tab:reshaping-vsp} shows the numerical results on the graph
vsp\_vibrobox\_scagr7-2c\_rlfddd, using the $\esp_H$ constructed via
spanning trees. This is a larger graph with 77328 vertices and 435586
edges. The second eigenvalue of the standard graph Laplacian is
$C_P=0.3033$, and on average each vertex has 11.3 neighbors. We can
also see huge reduction in the number of aggregates for achieving the
same approximation error. We point out that for all cases, the effect
of reduction becomes more significant when it comes to the finer
levels (top of the tables), since this allows for more space for
improvement.

\begin{table}
  \centering
  \begin{tabular}{l|r|c|c}
    Aggregation           & $n_c$   & $\eta/\lVert u_h-u\rVert_A$  & $\lVert u_h-u\rVert_A/\lVert u\rVert_A$ \\\hline
    $k_0=3$               & $3582$  & $1.1339$                     & $0.8463$                      \\
    $k_1=7$ then reshape  & $611$   & $1.3409$                     & $0.8144$                      \\\hline
    $k_0=4$               & $1784$  & $1.1466$                     & $0.8926$                      \\
    $k_1=7$ then reshape  & $400$   & $1.3013$                     & $0.8833$                      \\\hline
    $k_0=5$               & $890$   & $1.1934$                     & $0.9187$                      \\
    $k_1=7$ then reshape  & $341$   & $1.2858$                     & $0.9076$                      \\\hline
    $k_0=6$               & $444$   & $1.2514$                     & $0.9438$                      \\
    $k_1=7$ then reshape  & $298$   & $1.2754$                     & $0.9235$                      \\\hline
  \end{tabular}
  \caption{Comparison of direct matching and reshaping on the graph
    vsp\_vibrobox\_scagr7-2c\_rlfddd, at different levels of
    coarsening. The method using $\esp_H$ via spanning trees is
    reported here.}
  \label{tab:reshaping-vsp}
\end{table}

\subsection{Experiment II: point source example}
We further test the reshaping algorithm with the exact solution $u$
chosen to be the indicator function of some vertex (and the right-hand
side $f$ is set accordingly), and see if the aggregates can be adapted
to the solution. In figure \ref{fig:reshaping-ind} the chosen vertex
is marked by a red dot. The reshaping algorithm starts with a coarse
aggregation with number of aggregates $n_c=2$ (figure
\ref{fig:reshaping-ind1}) and ends at $n_c=15$ (figure
\ref{fig:reshaping-ind2}). We can see that the algorithm is adaptive
to the solution in the sense that it only selects those aggregates
around the chosen vertex to split, thus approximating the ``point
source'' solution using very few aggregates.

\begin{figure}
  \begin{subfigure}{0.49\textwidth}
    \includegraphics[width=\linewidth]{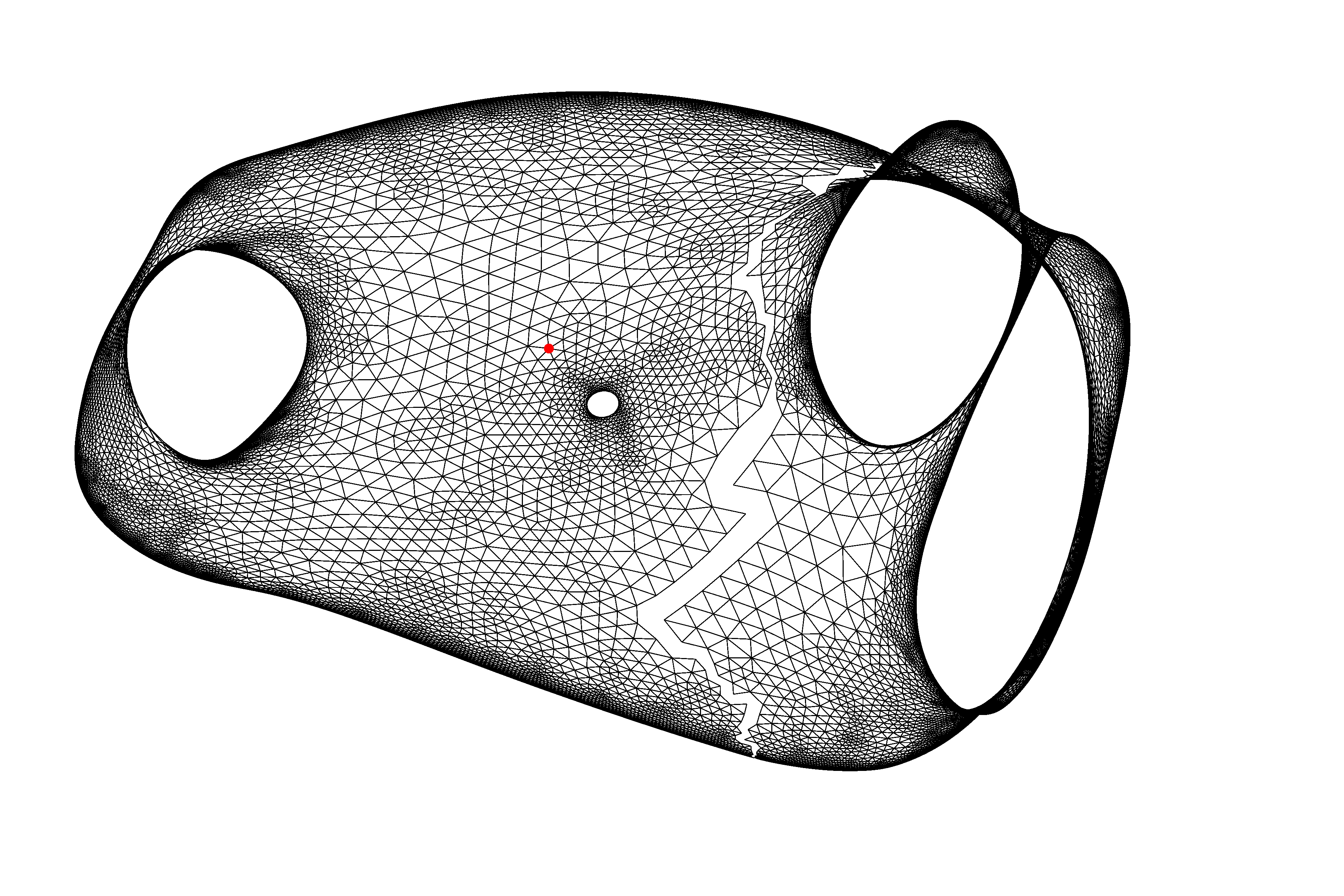}
    \caption{}\label{fig:reshaping-ind1}
  \end{subfigure}
  \begin{subfigure}{0.49\textwidth}
    \includegraphics[width=\linewidth]{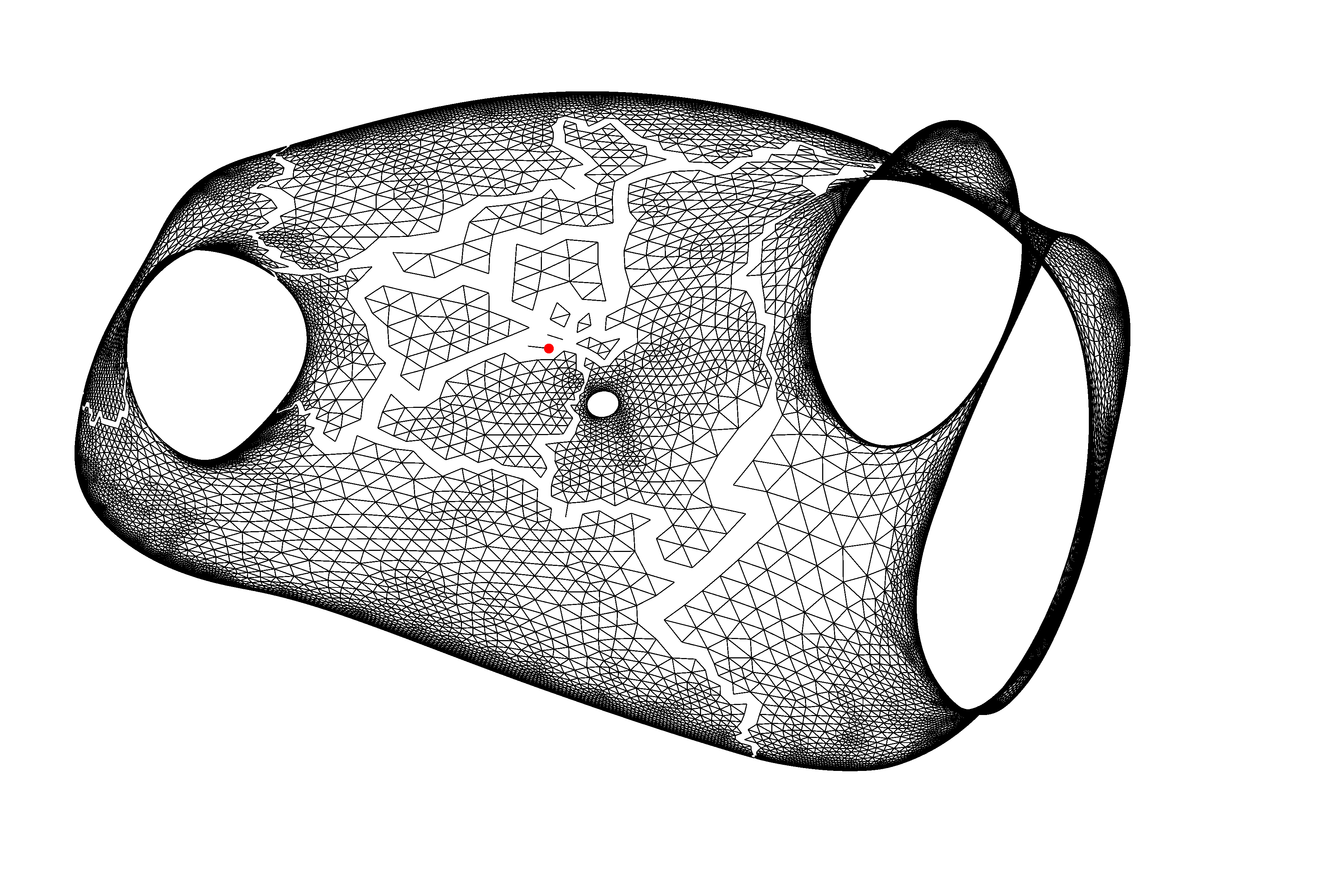}
    \caption{}
    \label{fig:reshaping-ind2}
  \end{subfigure}
  \caption{Reshaping algorithm adapted to a $u$ chosen to be the
    indicator function of a vertex.}
  \label{fig:reshaping-ind}
\end{figure}

\section{Conclusions and outlook}
In this paper we proposed an aggregation based multilevel hierarchy on
graphs for coarsening both the vertex and edge spaces. This hierarchy
is used together with the hypercircle a posteriori error estimates to
give a reliable bound on the error of approximate solutions to the
graph Laplacian equation. We introduced the practical algorithm for
reshaping of a given aggregation based on the hypercircle error
estimates, and the generated aggregations were compared with a
matching algorithm. We observed improvements in the quality of
aggregations using our reshaping algorithm.

Some further work may be done on studying how different subspaces
$\esp_H$ can be chosen for right-hand side $f$ of various properties,
as well as other mechanisms for splitting the aggregates chosen in the
reshaping algorithm.

\bibliographystyle{plain}
\bibliography{bib_adaptive}

\end{document}